\documentclass{amsart}
\usepackage{amsfonts,amsmath,amssymb}
\usepackage{geometry}
\newtheorem{theorem}{Theorem}
\newtheorem{proposition}{Proposition}
\newtheorem{lemma}{Lemma}
\newtheorem{corollary}{Corollary}
\newtheorem{definition}{Definition}

\begin{document}

\title{KPP type reaction--diffusion equations
       in narrow tubes with degenerating boundary and generalized second order operators}

\author{Hyejin Kim} 
\address{Department of Mathematics and Statistics, University of Michigan--Dearborn, Dearborn, MI 48128, USA}
\email{khyejin@umich.edu}



\maketitle

\begin{abstract}
We consider a reaction--diffusion equation with
Kolmogorov--Petrovskii--Piskunov (KPP) nonlinearity and Neumann boundary
conditions in a narrow tube
$G^{\varepsilon,\delta}\subset\mathbb{R}\times\mathbb{R}^{n}$ of width of order
$\varepsilon$.  The tube is chosen from a family indexed by a second parameter
$\delta>0$, whose boundary $\partial G^{\delta}$ is smooth for each $\delta>0$
but degenerates as $\delta\downarrow0$.  In the iterated limit
$\varepsilon\downarrow0$ followed by $\delta\downarrow0$ the problem becomes a
Cauchy problem on the line whose spatial operator is a generalized second
order differential operator $D_{v}D_{u}$ in the sense of Feller, rather than a
classical second order elliptic operator.  Solutions of the limiting problem
need not be differentiable in $x$, so they must be understood in a generalized
sense.  Using the Feynman--Kac representation, we prove that the limiting
problem has a unique generalized solution.  When the cross sectional volume is
bounded above and below, we also prove that this solution is
H\"older continuous in $t$ with exponent $1/2$, locally uniformly in $x$,
whenever it is Lipschitz continuous in $x$.  Under the same volume bounds, this Lipschitz condition holds
when the initial function and the nonlinearity are Lipschitz continuous in $x$.  Finally we prove that the solutions in the narrow
tubes converge to the generalized solution uniformly on compact time intervals
in the same iterated limit if the tube, the initial function, and the
nonlinearity are periodic in $x$.
\end{abstract}

\section{Introduction}

Consider a Wiener process with normal reflection in a tube in
$\mathbb{R}\times\mathbb{R}^{n}$ that is unbounded along the axis $\mathbb{R}$ and has
cross sections of diameter of order $\varepsilon$.  Its two components evolve on
different time scales.  The transverse component is confined to a cross section of
diameter $O(\varepsilon)$ and meets the boundary repeatedly, so its distribution
relaxes to equilibrium on the cross section in time $O(\varepsilon^{2})$, whereas the
longitudinal component is unconstrained and covers distances of order one in time of
order one.  On a fixed time interval the transverse component therefore equilibrates
ever more often as $\varepsilon\downarrow0$.  It is the fast variable, fast in
relaxation rather than in displacement, its excursions being of size $O(\varepsilon)$
only.  In the limit the transverse component enters only through an average, and it
enters as a drift.  Reflection at a boundary that is nowhere parallel to the axis has
a component along the axis, directed towards the wider parts of the tube.
Integrating this effect over the cross section produces the drift
$\frac{1}{2}\frac{d\ln V}{dx}(x)$, where $V(x)$ is the volume of the cross section at
$x$.  The limiting process on the axis is the one dimensional diffusion
$dX_{t}=dW_{t}+\frac{1}{2}\frac{d\ln V}{dx}(X_{t})\,dt$.  This dimension reduction
goes back to Freidlin and Wentzell~\cite{FW93} (see also
\cite[Chapter~7]{F96} and \cite{FW12}), and it converts an initial--boundary
value problem in a thin multidimensional tube into a Cauchy problem on the line.
When the reaction term is nonlinear, the reduction carries over through the
Feynman--Kac formula.  The solution is an expectation of a functional of the
trajectories of the associated process, so convergence of solutions follows from weak
convergence of the processes together with a fixed point argument.  Freidlin and
Spiliopoulos \cite{FS} established this for narrow domains with nonlinear boundary
conditions and smooth boundary.

The situation changes when the limiting tube is not smooth.  Spiliopoulos \cite{SP09}
studied narrow tubes whose rescaled cross sectional volume converges, as the width
goes to zero, to a non-smooth function built from smooth pieces, steps, and Dirac
masses.  He showed that the longitudinal process converges to a Markov process with
gluing conditions at the discontinuities.  In his setting the narrowing of the tube
and the degeneration of its cross section occur in a single limit.  In the present
paper they are controlled by two separate limits.  Let
$G^{\delta}\subset\mathbb{R}\times\mathbb{R}^{n}$, $\delta>0$, be a family of domains
with smooth boundary, let $V^{\delta}(x)$ denote the volume of the cross section
$G^{\delta}_{x}$, and let
$G^{\varepsilon,\delta}=\{(x,y):x\in\mathbb{R},\ y\varepsilon^{-1}\in G^{\delta}_{x}\}$
be the corresponding narrow tube of width of order $\varepsilon$.  We consider the
second initial--boundary value problem for a KPP type reaction--diffusion equation in
$G^{\varepsilon,\delta}$ and let the two parameters go to zero in turn.  For each
fixed $\delta>0$ the boundary is smooth, and as $\varepsilon\downarrow0$ the solution
converges to that of a one dimensional reaction--diffusion equation with drift
$\frac{1}{2}\frac{d\ln V^{\delta}}{dx}$.  However, as $\delta\downarrow0$, the
boundary $\partial G^{\delta}$ degenerates and the drift term need not converge to a
function.  What converges instead is the pair consisting of the scale function and
the speed measure,
\[
  u(x)=\lim_{\delta\downarrow0}\int_{0}^{x}\frac{dy}{V^{\delta}(y)},
  \qquad
  v(x)=\lim_{\delta\downarrow0}\int_{0}^{x}2V^{\delta}(y)\,dy ,
\]
and the limiting process is generated by the generalized second order differential
operator $D_{v}D_{u}$ of Feller \cite{WF1,WF2}.  The limiting
Cauchy problem is 
\[
  \frac{\partial w}{\partial t}=D_{v}D_{u}w+f(x,0,w),
  \qquad w(0,x)=g(x),
\]
and its solution need not be twice differentiable in $x$, or even differentiable.
Therefore, a weaker notion of solution is needed.  We call $w$ a
generalized solution if it satisfies the integral equation that the
Feynman--Kac formula gives for classical solutions, with the expectation taken with
respect to the Markov process governed by $D_{v}D_{u}$.  In particular, every
classical solution is a generalized solution.

This paper proves the well-posedness of this limiting problem, the
regularity of its solution, and the convergence of the solutions in the narrow tubes
to it.  We prove three results.  First, the limiting Cauchy problem with
KPP nonlinearity has a unique generalized solution
(Theorem~\ref{thm:existence}).  
Second, when the
cross sectional volume is bounded above and below, this generalized solution is
H\"older continuous in $t$ with exponent $1/2$, locally uniformly in $x$, whenever it
is Lipschitz continuous in $x$ (Proposition~\ref{pro:holder}).  Under the same volume bounds, the Lipschitz
bound holds when the initial function and the nonlinearity are
Lipschitz continuous in $x$ (Lemma~\ref{lem:lipschitz} and Corollary~\ref{cor:holder}).  
Third, when the tube, the initial function, and the nonlinearity are periodic in $x$,
the solutions $w^{\varepsilon,\delta}$ in the narrow tubes converge to $w$ uniformly
on compact time intervals, first as $\varepsilon\downarrow0$ with $\delta>0$ fixed
and then as $\delta\downarrow0$ (Proposition~\ref{pro:convergence}).

The probabilistic ingredients we import are the weak convergence results for
reflected Wiener processes in narrow tubes.  For fixed $\delta>0$ and smooth
boundary these are due to Freidlin \cite{F96} and Freidlin and Spiliopoulos
\cite{FS}.  For the degenerate limit we use the theorem of Freidlin and Wentzell
\cite{FW2} on weak convergence of one dimensional Markov processes, which
characterizes the limit through convergence of the scale function and the speed
measure.  The one dimensional theory of generalized second order operators and the
processes they generate is Feller's \cite{WF1,WF2}, with the accounts in Dynkin
\cite{EBD3} and Mandl \cite{PM1}.

The proof of Theorem~\ref{thm:existence} follows the fixed point argument for
Feynman--Kac integral equations in \cite[Chapter~5]{F85}, and the same
notion of generalized solution for KPP type equations with a Feller operator
$D_{v}D_{u}$ appears in Spiliopoulos \cite[Section~4]{SP12}.  For strips in $\mathbb{R}^{2}$, Kim
\cite[Theorem~4]{Kim09} proved the iterated limit for the processes, together
with the convergence of solutions of the linear Cauchy problem.  The present paper
treats the nonlinear problem.  It proves convergence of solutions of the KPP equation
through the iterated limit, uniformly on $[0,T]\times\mathbb{R}$ in the periodic
setting (Proposition~\ref{pro:convergence}), where the nonlinearity couples the
Feynman--Kac functionals to the unknown solution.  It also proves regularity of the
generalized solution in $x$ and $t$, with constants controlled by the cross sectional
volume and uniform in $\delta$ (Proposition~\ref{pro:holder},
Lemma~\ref{lem:lipschitz}, and Corollary~\ref{cor:holder}).  Lastly, it allows
cross sections of any dimension $n\ge1$ in the first limit, and it formulates the
question of front speeds in such tubes as an open problem
(Section~\ref{sec:remarks}).

The original motivation for this analysis was the asymptotic speed of wave front
propagation.  For KPP nonlinearities in periodic media the front speed is determined
by the principal eigenvalues of a family of twisted operators~\cite{BHN05, FG79}, and one would like to
know that these eigenvalues, and hence the speed, pass to the limit as
$\varepsilon\downarrow0$ and $\delta\downarrow0$.  That question is genuinely harder
than the ones settled here, because the eigenvalue is determined by the behavior as
$t\to\infty$, while the
weak convergence of the underlying processes holds for each fixed $t$, and the two
limits must be interchanged.  Rather than address it here, we leave it as an open
problem in Section~\ref{sec:remarks}, together with what is known.  Front
propagation in narrow domains with smooth boundary is treated in \cite{FS}, and the
case of narrow random channels, by way of an approximation on random graphs and large
deviation analysis, by Freidlin and Hu \cite{FH13}.

The paper is organized as follows.  Section~\ref{sec:setup} introduces the geometry, states the main hypotheses, and reviews the Feynman--Kac representations for the
problems in $G^{\varepsilon,\delta}$ and in the reduced domain.
Section~\ref{sec:generalized} contains the existence and uniqueness theorem, the
regularity results, and the convergence proposition.
Section~\ref{sec:remarks} closes with further remarks and an open problem.

\section{Narrow domains and the limiting problem}\label{sec:setup}

Let $G^{\delta}$, for $\delta>0$, be a domain in $\mathbb{R}\times \mathbb{R}^n$, $n\ge1$, with $\mathbb{R}\times\{0\} \subset G^{\delta}$ such that, for each $x\in \mathbb{R}$, the set $G_x^{\delta}=\{y\in \mathbb{R}^n:(x,y)\in G^{\delta}\}$ is a bounded connected domain, and assume that $\sup_{\delta>0}\sup_{x}\operatorname{diam}G^{\delta}_{x}<\infty$. The boundary $\partial G^{\delta}$ of $G^{\delta}$ is of class $C^{3+\alpha}$ for some $\alpha\in(0,1)$, as in \cite{FS}, and the inward unit normal
$\gamma^{\delta}(x,y)$ at every point of $\partial G^{\delta}$ is not parallel to
$\mathbb{R}$. We denote by $V^{\delta}(x)$ the volume of $G_x^{\delta}$ in $\mathbb{R}^n$, and for fixed $\delta>0$ we assume $V^{\delta} \in
C^{1}(\mathbb{R})$ with $V^{\delta}>0$. The narrow tube is $G^{\varepsilon,\delta}=\{(x,y):x\in\mathbb{R},\ y\varepsilon^{-1}\in G^{\delta}_{x}\}$
with $0<\varepsilon\ll1$.

Consider the second initial--boundary value problem in a narrow tube $G^{\varepsilon,\delta}$
\begin{equation}\label{equ:EQ1}
  \left\{\begin{array}{ll}
    \dfrac{\partial w^{\varepsilon,\delta}(t,x,y)}{\partial t}
      =\dfrac{1}{2}\Delta w^{\varepsilon,\delta}+f(x,y,w^{\varepsilon,\delta}),
      & t>0,\ (x,y)\in G^{\varepsilon,\delta},\\[2ex]
    \left.\dfrac{\partial w^{\varepsilon,\delta}(t,x,y)}{\partial\gamma^{\varepsilon,\delta}}
      \right|_{(x,y)\in\partial G^{\varepsilon,\delta}}=0,\quad t>0,
      & w^{\varepsilon,\delta}(0,x,y)=g(x),
  \end{array}\right.
\end{equation}
where $\Delta$ is the Laplacian in $(x,y)$ and $\gamma^{\varepsilon,\delta}$ is the
inward unit normal to $\partial G^{\varepsilon,\delta}$.  We fix $T>0$ and work on $[0,T]$ in this paper.  
Throughout the paper we assume that $f:\mathbb{R}\times\mathbb{R}^{n}\times\mathbb{R}\to\mathbb{R}$ is
continuous and bounded, and is of Kolmogorov--Petrovskii--Piskunov (KPP) type
\cite{KPP}, that is,
\[
  f(x,y,0)=f(x,y,1)=0,\qquad
  f(x,y,w)>0 \ \text{ for } w\in(0,1),\qquad
  f(x,y,w)<0 \ \text{ for } w\notin[0,1].
\]

For $w>0$ we write $f(x,y,w)=c(x,y,w)\,w$.  The function $c$ is continuous in all
variables, bounded, and Lipschitz continuous in $w$ uniformly in $(x,y)$, with
Lipschitz constant $K_{c}$.  We set
\[
  \Lambda:=\sup_{x,y,\,0\le w\le 1}c(x,y,w)<\infty,
  \qquad
  c_{0}(x,y):=c(x,y,0),
\]
so that $c_{0}$ is the linearization of the nonlinearity at $w=0$.  We assume the
KPP condition $c_{0}(x,y)=\max_{0\le w\le1}c(x,y,w)$.  Since
$f(x,y,w)<0$ for $w>1$, we have $c(x,y,w)<0$ there, and hence $c\le\Lambda$ for every
$w\ge0$, not only on $[0,1]$.  Since $f>0$ on $(0,1)$, we also have $c>0$ there,
so $\Lambda>0$ and $e^{\Lambda T}>1$.  The KPP condition is
what makes the linearized problem govern the behavior of the nonlinear one.  The
function $c_{0}$ is the potential in the eigenvalue problems referred to in
Section~\ref{sec:remarks}. The function $g:\mathbb{R}\to[0,1]$ is continuous. By \cite{AF}, problem
\eqref{equ:EQ1} has a unique bounded classical solution.

Let $(X^{\varepsilon,\delta}_{t},Y^{\varepsilon,\delta}_{t})$ be the Wiener process
in $G^{\varepsilon,\delta}$ with instantaneous normal reflection on
$\partial G^{\varepsilon,\delta}$. Its trajectories are described by the stochastic differential equations
\begin{equation}\label{equ:EQ3}
  \begin{aligned}
    dX^{\varepsilon,\delta}_{t}&=dW^{1}_{t}
      +\gamma^{\varepsilon,\delta}_{1}
       (X^{\varepsilon,\delta}_{t},Y^{\varepsilon,\delta}_{t})\,dL^{\varepsilon,\delta}_{t},\\
    dY^{\varepsilon,\delta}_{t}&=dW^{2}_{t}
      +\gamma^{\varepsilon,\delta}_{2}
       (X^{\varepsilon,\delta}_{t},Y^{\varepsilon,\delta}_{t})\,dL^{\varepsilon,\delta}_{t},
  \end{aligned}
\end{equation}
where $W^{1}_t$ and $W^{2}_t$ are independent Wiener processes in $\mathbb{R}$ and
$\mathbb{R}^{n}$, respectively, and $\gamma^{\varepsilon,\delta}_{1}(x,y)$ and
$\gamma^{\varepsilon,\delta}_{2}(x,y)$ are the projections of the inward unit normal onto
$\mathbb{R}$ and $\mathbb{R}^{n}$. Moreover, $L^{\varepsilon,\delta}_{t}$ is the local
time of the process $(X_t^{\varepsilon, \delta}, Y_t^{\varepsilon, \delta})$ on $\partial G^{\varepsilon,\delta}$. The local time is a continuous nondecreasing process that increases only on $\{t:(X^{\varepsilon,\delta}_{t},Y^{\varepsilon,\delta}_{t})
\in\partial G^{\varepsilon,\delta}\}$.  By the Feynman--Kac formula~\cite{F85}, the solution of
\eqref{equ:EQ1} has the following representation
\begin{equation}\label{equ:EQ5}
  w^{\varepsilon,\delta}(t,x,y)
  =E^{\varepsilon,\delta}_{x,y}\left[g(X^{\varepsilon,\delta}_{t})
    \exp\left\{\int_{0}^{t}
      c\bigl(X^{\varepsilon,\delta}_{s},Y^{\varepsilon,\delta}_{s},
        w^{\varepsilon,\delta}(t-s,X^{\varepsilon,\delta}_{s},Y^{\varepsilon,\delta}_{s})
      \bigr)ds\right\}\right].
\end{equation}
We write $\|h\|=\sup|h|$ for the supremum
norm of a bounded function $h$.  Since $0\le g\le1$ and $c\le\Lambda$ for every
$w\ge0$, \eqref{equ:EQ5} gives at once the bound
\begin{equation}\label{equ:unifbound}
  0\ \le\ w^{\varepsilon,\delta}(t,x,y)\ \le\ \|g\|\,e^{\Lambda T},
  \qquad 0\le t\le T,
\end{equation}
uniformly in $\varepsilon$ and $\delta$.

Under these assumptions, it is shown in \cite{F96,FS} that, for fixed $\delta>0$, as $\varepsilon\downarrow0$, the longitudinal
component $X^{\varepsilon,\delta}_{t}$, $0\le t\le T$, converges weakly in
$C([0,T];\mathbb{R})$ to the one dimensional diffusion $X^{\delta}_{t}$ that satisfies the stochastic differential equation
\begin{equation}\label{equ:EQ4}
  dX^{\delta}_{t}=dW^{1}_{t}+b^{\delta}(X^{\delta}_{t})\,dt,
  \qquad
  b^{\delta}(x)=\frac{1}{2}\,\frac{d\ln V^{\delta}(x)}{dx}.
\end{equation}
Because $|Y^{\varepsilon,\delta}_{s}|\le\varepsilon\,
\sup_{x}\operatorname{diam}G^{\delta}_{x}$ for all $s$, the transverse component
converges to $0$ uniformly on $[0,T]$, so the pair
$(X^{\varepsilon,\delta},Y^{\varepsilon,\delta})$ converges weakly in
$C([0,T];\mathbb{R}\times\mathbb{R}^{n})$ to $(X^{\delta},0)$.

Correspondingly, $w^{\varepsilon,\delta}$ is expected to converge to the solution
$w^{\delta}$ of the reduced problem
\begin{equation}\label{equ:EQ2}
  \left\{\begin{array}{ll}
    \dfrac{\partial w^{\delta}(t,x)}{\partial t}
      =\dfrac{1}{2}w^{\delta}_{xx}
       +\dfrac{1}{2}\dfrac{d\ln V^{\delta}(x)}{dx}\,w^{\delta}_{x}
       +c\bigl(x,0,w^{\delta}\bigr)\,w^{\delta},
      & t>0,\ x\in\mathbb{R},\\[2ex]
    w^{\delta}(0,x)=g(x), & x\in\mathbb{R}.
  \end{array}\right.
\end{equation}
Using the Feynman--Kac formula~\cite{F85}, we can write the solution $w^{\delta}(t, x)$ of \eqref{equ:EQ2} as the expectation of a functional of the trajectories of the corresponding process
\begin{equation}\label{equ:EQ6}
  w^{\delta}(t,x)=E^{\delta}_{x}\left[g(X^{\delta}_{t})
    \exp\left\{\int_{0}^{t}c\bigl(X^{\delta}_{s},0,
      w^{\delta}(t-s,X^{\delta}_{s})\bigr)ds\right\}\right].
\end{equation}
For fixed $\delta>0$ the coefficients of \eqref{equ:EQ2} are smooth by the
assumptions on $G^{\delta}$ and the equation is uniformly parabolic, so by
Friedman \cite{AF} problem \eqref{equ:EQ2} has a unique bounded classical solution,
and that solution is continuous and bounded on $[0,T]\times\mathbb{R}$.  This regularity is used in
Section~\ref{sec:generalized}.

Next, as $\delta\downarrow0$, the boundary $\partial G^{\delta}$ degenerates and the drift
$b^{\delta}$ need not converge to a function.  For fixed $\delta>0$, the scale function and speed
measure associated with \eqref{equ:EQ2} are
$u^{\delta}(x)=\int_{0}^{x}1/V^{\delta}(y)\,dy$ and
$v^{\delta}(x)=\int_{0}^{x}2V^{\delta}(y)\,dy$, and the generator of \eqref{equ:EQ2}
satisfies $\mathcal{L}^{\delta}=D_{v^{\delta}}D_{u^{\delta}}$, the associated
generalized second order differential operator in the sense of Feller \cite{WF1}. We assume that the limits
\[
  u(x)=\lim_{\delta\downarrow0}\int_{0}^{x}\frac{dy}{V^{\delta}(y)},
  \qquad
  v(x)=\lim_{\delta\downarrow0}\int_{0}^{x}2V^{\delta}(y)\,dy
\]
exist, the first for every $x\in\mathbb{R}$ and the second for every continuity
point $x$ of $v$, so that $u^{\delta}\to u$ and
$v^{\delta}\to v$ as $\delta\downarrow0$.  Here $u$ is continuous and strictly
increasing, $v$ is right continuous and strictly increasing, and both are finite on
compact subsets of $\mathbb{R}$, so $u$ and $v$ are a scale function and a
speed measure of the limiting process.  These
assumptions are exactly the conditions of the weak convergence theorem of Freidlin
and Wentzell \cite{FW2} (see also \cite[Section~1]{Kim09}).  They hold in particular
if both limits exist locally uniformly, in which case $v$ is continuous.  Since the
functions $u^{\delta}$ are increasing and $u$ is continuous, $u^{\delta}\to u$ locally
uniformly.  Because the cross sections have
uniformly bounded diameter, $V^{\delta}\le C$ for a constant $C$ independent of $x$
and $\delta$, so that $u^{\delta}(x)-u^{\delta}(y)\ge(x-y)/C$ and hence
$u(x)-u(y)\ge(x-y)/C$ for $x>y$.  Therefore $u^{\delta}$ and $u$ tend to $\pm\infty$
at $\pm\infty$, the boundaries $\pm\infty$ are inaccessible for $X^{\delta}$ and for
$X$, and these processes do not explode \cite{PM1}.  The theorem of Freidlin and Wentzell
\cite{FW2} then identifies the limit of
$X^{\delta}_{t}$ as the Markov process $X_{t}$ generated by the Feller operator
$D_{v}D_{u}$.  For strips in $\mathbb{R}^{2}$, that is for $n=1$, the combined
statement that $X^{\varepsilon,\delta}$ converges weakly to $X$ as first
$\varepsilon\downarrow0$ and then $\delta\downarrow0$ is Theorem~4 of Kim
\cite{Kim09}.  The same scheme, in which a non-smooth channel is approximated
by smooth ones and the limit is identified through \cite{FW2}, is used by Freidlin
and Hu \cite{FH13b} for narrow random channels.  The limiting Cauchy
problem is
\begin{equation}\label{equ:EQ7}
  \left\{\begin{array}{ll}
    \dfrac{\partial w(t,x)}{\partial t}=D_{v}D_{u}w+f(x,0,w),
      & t>0,\ x\in\mathbb{R},\\[1ex]
    w(0,x)=g(x), & x\in\mathbb{R}.
  \end{array}\right.
\end{equation}
A solution of \eqref{equ:EQ7} need not be twice differentiable in $x$, or even differentiable, so the
classical notion of solution is not available, and we use the following.

\begin{definition}\label{def:gensol}
A bounded nonnegative Borel function $w:[0,T]\times\mathbb{R}\to\mathbb{R}$, continuous in $x$ for each
$t$, is a \emph{generalized solution} of \eqref{equ:EQ7} if $w(0,\cdot)=g$ and
\begin{equation}\label{equ:EQ9}
  w(t,x)=E_{x}\left[g(X_{t})\exp\left\{\int_{0}^{t}
    c\bigl(X_{s},0,w(t-s,X_{s})\bigr)ds\right\}\right],
  \qquad 0\le t\le T,\ x\in\mathbb{R},
\end{equation}
where $X_{t}$ is the Markov process generated by $D_{v}D_{u}$ in the sense of
Feller \cite{WF1}.  The generalized solution of \eqref{equ:EQ2} is defined in the same way, with
$D_{v^{\delta}}D_{u^{\delta}}$ in place of $D_{v}D_{u}$, as a bounded
nonnegative Borel function satisfying \eqref{equ:EQ6}, where $X^{\delta}$ is
given by \eqref{equ:EQ4}.
\end{definition}

If a classical solution of \eqref{equ:EQ7} exists, it is a generalized solution in the
sense of Definition~\ref{def:gensol}.  The Feynman--Kac formula applied to the
classical solution yields exactly \eqref{equ:EQ9}.  
Theorem~\ref{thm:existence} in Section~\ref{sec:generalized} shows that \eqref{equ:EQ9} has a unique bounded solution
whether or not a classical solution exists.

\section{Generalized solutions and convergence}
\label{sec:generalized}

We assume the hypotheses of Section~\ref{sec:setup} throughout the rest of the paper
and do not repeat them in the statements below.  Throughout, $X_{t}$ is the Markov
process governed by the Feller operator $D_{v}D_{u}$.

Let $\mathcal{C}_{T}$ be the set of bounded nonnegative Borel functions
$\varphi:[0,T]\times\mathbb{R}\to\mathbb{R}$ that are continuous in $x$ for each
$t\in[0,T]$, equipped with the norm
$\|\varphi\|_{T}:=\sup_{0\le t\le T}\sup_{x\in\mathbb{R}}|\varphi(t,x)|$.
For $\varphi\in\mathcal{C}_{T}$ define
\begin{equation}\label{equ:Fop}
  \mathcal{F}_{t}[\varphi](x)
  :=E_{x}\left[g(X_{t})\exp\left\{\int_{0}^{t}
     c\bigl(X_{s},0,\varphi(t-s,X_{s})\bigr)ds\right\}\right],
  \qquad 0\le t\le T,\ x\in\mathbb{R},
\end{equation}
so that, by Definition~\ref{def:gensol}, $w$ is a generalized solution of
\eqref{equ:EQ7} precisely when $w\in\mathcal{C}_{T}$ and
$w(t,\cdot)=\mathcal{F}_{t}[w]$ for every $t\in[0,T]$.

\begin{lemma}\label{lem:basic}
For all $\varphi,\psi\in\mathcal{C}_{T}$ and $0\le t\le T$,
\begin{equation}\label{equ:contraction}
  \sup_{x\in\mathbb{R}}
    \bigl|\mathcal{F}_{t}[\varphi](x)-\mathcal{F}_{t}[\psi](x)\bigr|
  \ \le\ K\int_{0}^{t}\ \sup_{z\in\mathbb{R}}
     \bigl|\varphi(r,z)-\psi(r,z)\bigr|\,dr,
\end{equation}
where $K:=\|g\|\,e^{\Lambda T}K_{c}$.  Moreover $\mathcal{F}_{t}$ sends
$\mathcal{C}_{T}$ into itself, with
$\sup_{x}|\mathcal{F}_{t}[\varphi](x)|\le\|g\|e^{\Lambda T}$ for every
$\varphi\in\mathcal{C}_{T}$.
\end{lemma}

\begin{proof}
Since the exponents in \eqref{equ:Fop} are at most $\Lambda T$, the bound
$|e^{a}-e^{b}|\le e^{\max\{a,b\}}|a-b|$ and the Lipschitz property of $c$ in $w$
give
\[
  \bigl|\mathcal{F}_{t}[\varphi](x)-\mathcal{F}_{t}[\psi](x)\bigr|
  \ \le\ \|g\|\,e^{\Lambda T}\,E_{x}\!\int_{0}^{t}
     \bigl|c(X_{s},0,\varphi(t-s,X_{s}))-c(X_{s},0,\psi(t-s,X_{s}))\bigr|ds
\]
\[
  \ \le\ \|g\|\,e^{\Lambda T}K_{c}\,E_{x}\!\int_{0}^{t}
     \bigl|\varphi(t-s,X_{s})-\psi(t-s,X_{s})\bigr|ds
  \ \le\ K\int_{0}^{t}\sup_{z\in\mathbb{R}}
     \bigl|\varphi(t-s,z)-\psi(t-s,z)\bigr|ds .
\]
The bound holds for every $x\in\mathbb{R}$ and its right hand side does not depend on
$x$, so the supremum over $x$ may be taken on the left, which gives
\eqref{equ:contraction} after the change of variable $r=t-s$.

Next we check that $\mathcal{F}_{t}[\varphi]$ is continuous in $x$ when
$\varphi\in\mathcal{C}_{T}$.  Fix $t$ and consider the functional
\[
  \Theta(\xi):=g(\xi_{t})\exp\left\{\int_{0}^{t}
    c\bigl(\xi_{s},0,\varphi(t-s,\xi_{s})\bigr)ds\right\},
  \qquad \xi\in C([0,t];\mathbb{R}).
\]
It is bounded by $\|g\|e^{\Lambda T}$ and is continuous.  Suppose
$\xi^{(m)}\to\xi$ uniformly on $[0,t]$.  Then $g(\xi^{(m)}_{t})\to g(\xi_{t})$, and
for each fixed $s$,
\[
  c\bigl(\xi^{(m)}_{s},0,\varphi(t-s,\xi^{(m)}_{s})\bigr)
  \ \longrightarrow\
  c\bigl(\xi_{s},0,\varphi(t-s,\xi_{s})\bigr),
\]
because $g$, $\varphi(t-s,\cdot)$, and $c$ are continuous.  As $|c|$ is bounded, the
integrals converge by dominated convergence, so
$\Theta(\xi^{(m)})\to\Theta(\xi)$.  The process generated by $D_{v}D_{u}$ is a Feller
process \cite{EBD3,WF1,WF2,PM1}, so the law of the path started at $x$ depends
continuously on $x$.  Hence $x\mapsto E_{x}\Theta$ is continuous for every bounded
continuous $\Theta$ on $C([0,t];\mathbb{R})$.  Since
$\mathcal{F}_{t}[\varphi](x)=E_{x}\Theta$ for the functional $\Theta$ above,
$\mathcal{F}_{t}[\varphi]$ is continuous in $x$.

Finally, $|\mathcal{F}_{t}[\varphi]|\le\|g\|e^{\Lambda T}$ is immediate from
$c\le\Lambda$ and $\|g\|<\infty$.  In addition, $\mathcal{F}_{t}[\varphi]\ge0$ since
$g\ge0$, and $(t,x)\mapsto\mathcal{F}_{t}[\varphi](x)$ is Borel, because the
functional in \eqref{equ:Fop} is Borel in $t$ and the path, and $x\mapsto P_{x}$ is a
measurable kernel.  Together with the continuity just established,
this shows $\mathcal{F}_{t}[\varphi]\in\mathcal{C}_{T}$.
\end{proof}

\begin{theorem}\label{thm:existence}
Problem \eqref{equ:EQ7} has a unique
generalized solution in the sense of Definition~\ref{def:gensol}.
\end{theorem}

\begin{proof}
First, we prove existence.  Define a sequence by
\begin{equation}\label{equ:EQ10}
  w_{0}(t,x):=g(x),
  \qquad
  w_{n+1}(t,\cdot):=\mathcal{F}_{t}[w_{n}]\quad (n\ge0).
\end{equation}
Since $g$ is continuous and bounded, $w_{0}\in\mathcal{C}_{T}$, and by
Lemma~\ref{lem:basic} each $w_{n}$ lies in $\mathcal{C}_{T}$ and satisfies
$\|w_{n}\|_{T}\le\|g\|e^{\Lambda T}$ for $n\ge1$.  Set
$d_{n}(t):=\sup_{x\in\mathbb{R}}|w_{n}(t,x)-w_{n-1}(t,x)|$ for $n\ge1$ and
$0\le t\le T$.  Applying \eqref{equ:contraction} to $\varphi=w_{n}$, $\psi=w_{n-1}$
gives
\begin{equation}\label{equ:recursion}
  d_{n+1}(t)\ \le\ K\int_{0}^{t}d_{n}(r)\,dr,
  \qquad n\ge1,\ 0\le t\le T .
\end{equation}

By the induction principle, we show that
\begin{equation}\label{equ:induction}
  d_{n}(t)\ \le\ C_{1}\,\frac{(Kt)^{n-1}}{(n-1)!},
  \qquad n\ge1,\ 0\le t\le T,
\end{equation}
where $C_{1}:=2\|g\|e^{\Lambda T}$.  For $n=1$, since $|w_{0}|\le\|g\|$ and
$|w_{1}|\le\|g\|e^{\Lambda T}$, we have
$d_{1}(t)\le\|g\|\bigl(1+e^{\Lambda T}\bigr)\le C_{1}$.  If
\eqref{equ:induction} holds for some $n\ge1$, then by \eqref{equ:recursion},
\[
  d_{n+1}(t)\ \le\ K\int_{0}^{t}C_{1}\frac{(Kr)^{n-1}}{(n-1)!}\,dr
  \ =\ C_{1}\frac{(Kt)^{n}}{n!},
\]
which is \eqref{equ:induction} for $n+1$.  Consequently
\[
  \sum_{k\ge1}\ \sup_{0\le t\le T}\sup_{x\in\mathbb{R}}
    \bigl|w_{k}(t,x)-w_{k-1}(t,x)\bigr|
  \ \le\ C_{1}\sum_{k\ge1}\frac{(KT)^{k-1}}{(k-1)!}
  \ =\ C_{1}e^{KT}\ <\ \infty ,
\]
so by the Weierstrass $M$-test the telescoping series
$w_{0}+\sum_{k\ge1}(w_{k}-w_{k-1})$ converges absolutely and uniformly on
$[0,T]\times\mathbb{R}$.  Its sum is $w:=\lim_{n\to\infty}w_{n}$.  As a
uniform limit of functions continuous in $x$, $w$ is continuous in $x$ for each $t$,
and $\|w\|_{T}\le\|g\|e^{\Lambda T}$, so $w\in\mathcal{C}_{T}$.  Also
$w(0,x)=\lim_{n}w_{n}(0,x)=g(x)$, since $w_{n}(0,\cdot)=g$ for every $n$ by
\eqref{equ:Fop}.

It remains to verify that $w$ is a fixed point of $\mathcal{F}$.  For fixed
$t\in[0,T]$ and every $n\ge0$,
\[
  \sup_{x}\bigl|w(t,x)-\mathcal{F}_{t}[w](x)\bigr|
  \ \le\ \sup_{x}\bigl|w(t,x)-w_{n+1}(t,x)\bigr|
        +\sup_{x}\bigl|\mathcal{F}_{t}[w_{n}](x)-\mathcal{F}_{t}[w](x)\bigr| ,
\]
using $w_{n+1}(t,\cdot)=\mathcal{F}_{t}[w_{n}]$.  The first term tends to $0$ as
$n\to\infty$ by uniform convergence.  By \eqref{equ:contraction} the second term is
at most $K\int_{0}^{t}\sup_{z}|w_{n}(r,z)-w(r,z)|\,dr\le KT\|w_{n}-w\|_{T}$, which
also tends to $0$.  Since the left hand side does not depend on $n$, letting $n\to\infty$ shows that
it is zero.  Hence $w(t,\cdot)=\mathcal{F}_{t}[w]$ for every $t\in[0,T]$, and $w$ is a
generalized solution of \eqref{equ:EQ7}.

\medskip
Next, we prove uniqueness.  Let $w$ and $\tilde w$ be generalized solutions in
$\mathcal{C}_{T}$ and set
\[
  l(t):=\sup_{0\le s\le t}\ \sup_{x\in\mathbb{R}}
    \bigl|w(s,x)-\tilde w(s,x)\bigr|,\qquad 0\le t\le T .
\]
Then $l$ is finite, nonnegative, and nondecreasing, and $l(0)=0$ because both
functions equal $g$ at $t=0$.  For $0\le s\le t$, Lemma~\ref{lem:basic} applied to
$w(s,\cdot)=\mathcal{F}_{s}[w]$ and $\tilde w(s,\cdot)=\mathcal{F}_{s}[\tilde w]$
gives
\[
  \sup_{x}\bigl|w(s,x)-\tilde w(s,x)\bigr|
  \ \le\ K\int_{0}^{s}\sup_{z}\bigl|w(r,z)-\tilde w(r,z)\bigr|dr
  \ \le\ K\int_{0}^{t}l(r)\,dr ,
\]
where the last step uses $\sup_{z}|w(r,z)-\tilde w(r,z)|\le l(r)$ and
$s\le t$.  The right hand side is independent of $s$, so taking the supremum over
$s\in[0,t]$,
\begin{equation}\label{equ:EQ14}
  l(t)\ \le\ K\int_{0}^{t}l(r)\,dr,\qquad 0\le t\le T .
\end{equation}
By Gr\"onwall's inequality, \eqref{equ:EQ14} gives $l(t)\le0\cdot e^{Kt}=0$ for all
$t\in[0,T]$, since $l$ is bounded and measurable.  Hence $w=\tilde w$ on
$[0,T]\times\mathbb{R}$.
\end{proof}

The bound $0\le w\le\|g\|e^{\Lambda T}$ used above is cruder than the bound
$0\le w\le1$ expected for KPP nonlinearities with $0\le g\le1$, but it is all the
fixed point argument requires.  The constants $K=\|g\|e^{\Lambda T}K_{c}$ and $C_{1}=2\|g\|e^{\Lambda T}$ depend on
$T$ but not on $\varepsilon$, $\delta$, $n$, or $x$.

\begin{proposition}\label{pro:holder}
Let $w$ be the generalized solution of \eqref{equ:EQ7} given by
Theorem~\ref{thm:existence}, and assume that $V^{\delta}$ is bounded above and below
on compact sets, uniformly in $\delta$.  Assume also that $w$ is Lipschitz
continuous in $x$, uniformly in $t$, that is, that there is $K_{w}<\infty$ with
\begin{equation}\label{equ:lipschyp}
  |w(t,x_{1})-w(t,x_{2})|\ \le\ K_{w}|x_{1}-x_{2}|
  \qquad\text{for all } x_{1},x_{2}\in\mathbb{R},\ 0<t\le T .
\end{equation}
Then $w$ is H\"older continuous in $t$ with exponent $1/2$, locally uniformly in $x$. 
For every compact $K\subset\mathbb{R}$ there is a constant $C=C(K,T)$ such that
\[
  |w(t+h,x)-w(t,x)|\ \le\ C\sqrt{h}
  \qquad\text{for all } x\in K,\ 0<t\le t+h\le T .
\]
\end{proposition}

\begin{proof}
Fix a compact $K\subset\mathbb{R}$.  By hypothesis there are constants
$0<m\le M<\infty$, depending on $K$ but not on $\delta$, with $m\le V^{\delta}\le M$
on $K$.  Since $du^{\delta}=dx/V^{\delta}$ and $dv^{\delta}=2V^{\delta}dx$, this
gives $du^{\delta}\ge M^{-1}dx$ and $dv^{\delta}\ge 2m\,dx$ on $K$.  Letting
$\delta\downarrow0$ and using the right continuity of $v$, we obtain
\[
  du(x)\ \ge\ a_{1}(K)\,dx
  \qquad\text{and}\qquad
  dv(x)\ \ge\ a_{2}(K)\,dx
  \qquad\text{on } K ,
\]
in the sense of measures, with $a_{1}(K)=M^{-1}$ and $a_{2}(K)=2m$.  Write
$b(K):=\min\{a_{1}(K),a_{2}(K)\}$.

Note that $\|w\|=\sup_{t,x}|w(t,x)|\le\|g\|e^{\Lambda T}$ by
Theorem~\ref{thm:existence}, and that $\|c\|=\sup|c|<\infty$ since $c$ is bounded.

First we reduce the assertion to a second moment estimate.
Fix $x\in\mathbb{R}$ and $h>0$ with $t+h\le T$.  By \eqref{equ:EQ9} and the Markov
property of $X_{t}$,
\begin{equation}\label{equ:EQ15}
  w(t+h,x)=E_{x}\left[w(t,X_{h})\exp\left\{\int_{0}^{h}
    c\bigl(X_{s},0,w(t+h-s,X_{s})\bigr)ds\right\}\right].
\end{equation}
Subtracting $w(t,x)$ and inserting $w(t,x)\cdot1$ inside the expectation,
\begin{equation}\label{equ:EQ16}
  \begin{aligned}
  |w(t+h,x)-w(t,x)|
  &\le E_{x}\Bigl[|w(t,X_{h})-w(t,x)|
       \exp\Bigl\{\int_{0}^{h}c\,ds\Bigr\}\Bigr]\\
  &\qquad+\|w\|\,E_{x}\Bigl|\exp\Bigl\{\int_{0}^{h}c\,ds\Bigr\}-1\Bigr|\\
  &\le e^{\|c\|T}\,E_{x}\bigl|w(t,X_{h})-w(t,x)\bigr|
      +\|w\|\,e^{\|c\|T}\|c\|\,h ,
  \end{aligned}
\end{equation}
because $|e^{a}-1|\le e^{|a|}|a|$ with $|a|\le\|c\|h\le\|c\|T$.
Since $h\le\sqrt{T}\sqrt{h}$ for $0<h\le T$, the second term is at most
$\|w\|e^{\|c\|T}\|c\|\sqrt{T}\sqrt{h}$.  It therefore suffices to bound
$E_{x}|w(t,X_{h})-w(t,x)|$ by a constant multiple of $\sqrt{h}$, and by
\eqref{equ:lipschyp} it suffices to control $E_{x}|X_{h}-x|$.

Next we construct an auxiliary function to which Dynkin's formula applies.
Write $x_{0}:=x$ for the point fixed above, and set
\begin{equation}\label{equ:EQ17}
  \Psi(x):=\int_{x_{0}}^{x}\int_{x_{0}}^{y}dv(z)\,du(y),
  \qquad x\in\mathbb{R}.
\end{equation}
Since $u$ and $v$ are finite on compact subsets of $\mathbb{R}$, the inner integral
is finite for every $y$ and $\Psi$ is finite and continuous on $\mathbb{R}$, with
$\Psi(x_{0})=0$.

The function $\Psi$ was chosen so that $D_{v}D_{u}\Psi\equiv1$.
Differentiating \eqref{equ:EQ17} with respect to $u$ removes the outer integral,
\[
  D_{u}\Psi(x)=\int_{x_{0}}^{x}dv(z)=v(x)-v(x_{0}),
\]
and differentiating this with respect to $v$ removes the remaining one, so
$D_{v}D_{u}\Psi\equiv1$.  In particular, $D_{v}D_{u}\Psi$ is defined everywhere.

The function $\Psi$ is the first member of the family obtained by iterating this
construction,
\[
  \Psi_{0}\equiv1,
  \qquad
  \Psi_{n}(x)=\int_{x_{0}}^{x}\int_{x_{0}}^{y}\Psi_{n-1}(z)\,dv(z)\,du(y)
  \quad(n\ge1),
\]
which satisfies $D_{v}D_{u}\Psi_{n}=\Psi_{n-1}$ for every $n$ and, by the same
argument as below, controls the higher moments of the increments of $X$.  Only
$\Psi=\Psi_{1}$ is needed here.

Moreover $\Psi\ge0$.  For $x>x_{0}$ both integrations in \eqref{equ:EQ17} run
forwards, so the integrand $v(y)-v(x_{0})$ is nonnegative and so is the integral.
For $x<x_{0}$ both run backwards, and reversing each of them changes the sign twice,
so again $\Psi(x)\ge0$.

For any bounded $f\in\mathcal{D}(D_{v}D_{u})$ the process
$f(X_{t})-\int_{0}^{t}D_{v}D_{u}f(X_{s})\,ds$ is a martingale
\cite{EBD3,PM1}.  Taking $f=\Psi$ formally and using $D_{v}D_{u}\Psi\equiv1$, the integral
reduces to $\int_{0}^{t}ds=t$, so
\begin{equation}\label{equ:EQ18}
  \Phi(t):=\Psi(X_{t})-t
\end{equation}
is a local martingale under $P_{x_{0}}$, with $\Phi(0)=\Psi(x_{0})=0$.
Since $\Psi$ is unbounded, it need not lie in the domain of the Feller
generator itself.  What we use is that $\Phi(t\wedge\sigma)$ is a martingale when
$\sigma$ is the exit time of $X$ from a bounded interval.  To see this, take a
bounded function $f\in\mathcal{D}(D_{v}D_{u})$ that coincides with $\Psi$ on a
neighborhood of the closed interval.  Up to time $\sigma$ the process stays in the
closed interval, where $f=\Psi$ and $D_{v}D_{u}f=1$, so stopping the martingale
$f(X_{t})-\int_{0}^{t}D_{v}D_{u}f(X_{s})\,ds$ at $\sigma$ gives exactly
$\Phi(t\wedge\sigma)$.

Let $R>0$, put $K_{R}:=[x_{0}-R,\,x_{0}+R]$, and let
\[
  \tau:=\inf\{s\ge0:\ |X_{s}-x_{0}|\ge R\} .
\]
By the preceding paragraph, $\Phi(t\wedge\tau)$ is a martingale, so
$E_{x_{0}}\Phi(h\wedge\tau)=0$, that is,
\begin{equation}\label{equ:optional}
  E_{x_{0}}\bigl[\Psi(X_{h\wedge\tau})\bigr]
  =E_{x_{0}}\bigl[h\wedge\tau\bigr]\ \le\ h .
\end{equation}

We now bound $\Psi$ from below.
Let $b:=b(K_{R})=\min\{a_{1}(K_{R}),a_{2}(K_{R})\}>0$.  For
$x\in K_{R}$ with $x>x_{0}$,
\[
  \Psi(x)=\int_{x_{0}}^{x}\!\!\int_{x_{0}}^{y}dv(z)\,du(y)
  \ \ge\ \int_{x_{0}}^{x}\!\!\int_{x_{0}}^{y}b^{2}\,dz\,dy
  \ =\ \frac{b^{2}}{2}(x-x_{0})^{2},
\]
using $dv\ge b\,dz$ and $du\ge b\,dy$ on $K_{R}$.  For $x\in K_{R}$ with $x<x_{0}$,
reversing both orientations,
\[
  \Psi(x)=\int_{x}^{x_{0}}\!\!\left(\int_{y}^{x_{0}}dv(z)\right)du(y)
  \ \ge\ \int_{x}^{x_{0}}b(x_{0}-y)\,b\,dy
  \ =\ \frac{b^{2}}{2}(x_{0}-x)^{2},
\]
and the bound is trivial at $x=x_{0}$.  Hence
\begin{equation}\label{equ:EQ20}
  \Psi(x)\ \ge\ \frac{b^{2}}{2}\,(x-x_{0})^{2}
  \qquad\text{for all }x\in K_{R}.
\end{equation}

This gives a second moment estimate for the stopped process.
Since $X_{h\wedge\tau}\in K_{R}$ almost surely, \eqref{equ:EQ20} and
\eqref{equ:optional} give
\begin{equation}\label{equ:secondmoment}
  E_{x_{0}}\bigl[(X_{h\wedge\tau}-x_{0})^{2}\bigr]\ \le\ \frac{2h}{b^{2}} .
\end{equation}
On $\{\tau\le h\}$ one has $|X_{\tau}-x_{0}|=R$, so \eqref{equ:secondmoment} also
yields the exit estimate
\begin{equation}\label{equ:exit}
  P_{x_{0}}(\tau\le h)\ \le\ \frac{2h}{b^{2}R^{2}} .
\end{equation}

Next, splitting the expectation appearing in \eqref{equ:EQ16} according to whether
the process has exited $K_{R}$ by time $h$ gives
\[
  E_{x_{0}}\bigl|w(t,X_{h})-w(t,x_{0})\bigr|
  \ \le\ K_{w}\,E_{x_{0}}\bigl[|X_{h}-x_{0}|\,\mathbf{1}_{\{\tau>h\}}\bigr]
       +2\|w\|\,P_{x_{0}}(\tau\le h).
\]
On $\{\tau>h\}$ we have $X_{h}=X_{h\wedge\tau}$, so by the Cauchy--Schwarz
inequality and \eqref{equ:secondmoment},
\[
  E_{x_{0}}\bigl[|X_{h}-x_{0}|\,\mathbf{1}_{\{\tau>h\}}\bigr]
  \ \le\ \Bigl(E_{x_{0}}\bigl[(X_{h\wedge\tau}-x_{0})^{2}\bigr]\Bigr)^{1/2}
  \ \le\ \frac{\sqrt{2}}{b}\sqrt{h},
\]
while by \eqref{equ:exit} and $h\le\sqrt{T}\sqrt{h}$,
\[
  2\|w\|\,P_{x_{0}}(\tau\le h)
  \ \le\ \frac{4\|w\|}{b^{2}R^{2}}\,h
  \ \le\ \frac{4\|w\|\sqrt{T}}{b^{2}R^{2}}\sqrt{h}.
\]
Substituting into \eqref{equ:EQ16},
\begin{equation}\label{equ:EQ21}
  |w(t+h,x_{0})-w(t,x_{0})|\ \le\ C\sqrt{h},
\end{equation}
with
\[
  C\ =\ e^{\|c\|T}\left(\frac{\sqrt{2}\,K_{w}}{b}
        +\frac{4\|w\|\sqrt{T}}{b^{2}R^{2}}\right)
      +\|w\|\,e^{\|c\|T}\|c\|\sqrt{T}.
\]
Finally, given a compact $K\subset\mathbb{R}$, take $R=1$ and apply the above at each
$x_{0}\in K$ with $b=b(K_{1})$, where $K_{1}:=\{x:\operatorname{dist}(x,K)\le1\}$ is
again compact.  The constant $C$ then depends only on $K$ and $T$, and
\eqref{equ:EQ21} holds for every $x_{0}\in K$.  This is the assertion.
\end{proof}

Lemma~\ref{lem:lipschitz} below shows that the Lipschitz bound \eqref{equ:lipschyp}
holds when the cross sectional volume is bounded above and below on all of
$\mathbb{R}$, uniformly in $\delta$, and $g$ and $c$ are Lipschitz in $x$.  Then the
conclusion of Proposition~\ref{pro:holder} holds without assuming
\eqref{equ:lipschyp} (Corollary~\ref{cor:holder}).  The bounds on $V^{\delta}$ in
Proposition~\ref{pro:holder} allow the limiting volume to have jumps, as for a tube
whose cross section changes abruptly.  In that case $b^{\delta}$ does not converge to
a function, and the limiting operator carries a gluing condition of skew type at each
jump, as in \cite{SP09} and \cite[Section~4, Example~(1)]{Kim09}.

\begin{lemma}\label{lem:lipschitz}
Assume that there are constants $0<m_{0}\le M_{0}<\infty$ with
$m_{0}\le V^{\delta}(x)\le M_{0}$ for all $x\in\mathbb{R}$ and $\delta>0$, that $g$
is Lipschitz continuous with constant $K_{g}$, and that $c(x,0,w)$ is Lipschitz
continuous in $x$ with constant $K_{0}$, uniformly in
$w\in[0,\|g\|e^{\Lambda T}]$.  Put $\kappa:=M_{0}/m_{0}$ and
\[
  A_{0}:=\kappa\,e^{\Lambda T}\bigl(K_{g}+\|g\|K_{0}T\bigr),
  \qquad
  B_{0}:=\kappa\,\|g\|e^{\Lambda T}K_{c}.
\]
Then the generalized solution $w$ of \eqref{equ:EQ7} satisfies
\eqref{equ:lipschyp} with $K_{w}=A_{0}e^{B_{0}T}$.  The same bound holds for the
solution $w^{\delta}$ of \eqref{equ:EQ2}, for every $\delta>0$.
\end{lemma}

\begin{proof}
First, we compare the scale function and the speed measure with Lebesgue measure.
For $x_{1}<x_{2}$ we have
$(x_{2}-x_{1})/M_{0}\le u^{\delta}(x_{2})-u^{\delta}(x_{1})\le(x_{2}-x_{1})/m_{0}$ and
$2m_{0}(x_{2}-x_{1})\le v^{\delta}(x_{2})-v^{\delta}(x_{1})\le2M_{0}(x_{2}-x_{1})$.  Letting
$\delta\downarrow0$ and using the right continuity of $v$, we obtain
\begin{equation}\label{equ:bilip}
  \frac{x_{2}-x_{1}}{M_{0}}\ \le\ u(x_{2})-u(x_{1})\ \le\ \frac{x_{2}-x_{1}}{m_{0}},
  \qquad
  2m_{0}(x_{2}-x_{1})\ \le\ v(x_{2})-v(x_{1})\ \le\ 2M_{0}(x_{2}-x_{1}).
\end{equation}
In particular $u$ and $v$ tend to $\pm\infty$ at $\pm\infty$, so $\pm\infty$ are
natural boundaries and $X$ does not explode~\cite{PM1}.

Next we show that $u(X_{t})$ is $P_{x}$-integrable with expectation $u(x)$ for every
$x\in\mathbb{R}$ and $0\le t\le T$.  Fix such $x$ and $t$, and let $\sigma_{n}$ be the
exit time of $X$ from $(x-n,x+n)$.  Since $D_{u}u\equiv1$
and hence $D_{v}D_{u}u\equiv0$, the martingale property used in the proof of
Proposition~\ref{pro:holder}, applied to a bounded function in
$\mathcal{D}(D_{v}D_{u})$ that coincides with $u$ near $[x-n,x+n]$, gives
$E_{x}u(X_{t\wedge\sigma_{n}})=u(x)$.  By \eqref{equ:bilip}, the lower bounds on $du$
and $dv$ used in the proof of Proposition~\ref{pro:holder} hold on all of
$\mathbb{R}$ with $b:=\min\{M_{0}^{-1},2m_{0}\}$, so \eqref{equ:secondmoment} with
$R=n$ gives $E_{x}(X_{t\wedge\sigma_{n}}-x)^{2}\le2t/b^{2}$ for every $n$.  Since
$|u(y)-u(x)|\le|y-x|/m_{0}$ by \eqref{equ:bilip}, the family
$\{u(X_{t\wedge\sigma_{n}})\}_{n}$ is bounded in $L^{2}$ and hence uniformly integrable.
As $X$ does not explode, $\sigma_{n}\to\infty$ and $u(X_{t\wedge\sigma_{n}})\to u(X_{t})$
almost surely, and therefore
\begin{equation}\label{equ:scalemart}
  E_{x}|u(X_{t})|<\infty,\qquad E_{x}u(X_{t})=u(x),
  \qquad x\in\mathbb{R},\ 0\le t\le T .
\end{equation}

We now construct an ordered coupling of the processes started at two points.  Fix
$x_{1}<x_{2}$, and let $X$ and $Y$ be independent copies of
the process, started at $x_{1}$ and at $x_{2}$.  The pair $(X,Y)$ is a Feller process on
$\mathbb{R}^{2}$, being a product of independent Feller processes, and is therefore
strong Markov with respect to the usual augmentation $(\mathcal{F}_{t})$ of its
natural filtration.  The coupling time
$\tau:=\inf\{t\ge0:\ X_{t}=Y_{t}\}$ is a stopping time, as the hitting time of the
closed diagonal by a continuous adapted process.  Define $Z_{t}:=Y_{t}$ for $t<\tau$
and $Z_{t}:=X_{t}$ for $t\ge\tau$.  We show that $Z$ has the law of the process started at $x_{2}$ and that
$X_{t}\le Z_{t}$ for all $t\ge0$.

To see that $Z$ has the law of the process started at $x_{2}$, let $\Gamma$ be the
measurable map that concatenates a path stopped at time $s$ with a path started at
the stopping point, given by $\Gamma(\zeta,s,z)_{r}:=\zeta_{r}$ for $r<s$ and
$\Gamma(\zeta,s,z)_{r}:=z_{r-s}$ for $r\ge s$.  Then
$Z=\Gamma(Y_{\cdot\wedge\tau},\tau,X_{\tau+\cdot})$ and
$Y=\Gamma(Y_{\cdot\wedge\tau},\tau,Y_{\tau+\cdot})$ on $\{\tau<\infty\}$, while
$Z=Y$ on $\{\tau=\infty\}$.  The pair $(Y_{\cdot\wedge\tau},\tau)$ is
$\mathcal{F}_{\tau}$-measurable, and by the strong Markov property of $(X,Y)$ at
$\tau$, conditionally on $\mathcal{F}_{\tau}$ the paths $X_{\tau+\cdot}$ and
$Y_{\tau+\cdot}$ have the laws $P_{X_{\tau}}$ and $P_{Y_{\tau}}$ on
$\{\tau<\infty\}$.  These coincide because $X_{\tau}=Y_{\tau}$.  Hence, for every
bounded measurable functional $G$ on path space,
$E\,G(Z)=E\,G(Y)$.

To see that $X_{t}\le Z_{t}$ for all $t\ge0$, note that for $t\ge\tau$ this is an
equality.
On $[0,\tau)$ the continuous function $Y-X$ does not vanish and is positive at
$t=0$, so it is positive there, that is, $X_{t}<Y_{t}=Z_{t}$.

By this ordering and \eqref{equ:bilip},
$Z_{t}-X_{t}\le M_{0}\bigl(u(Z_{t})-u(X_{t})\bigr)$.  Taking expectations, and using
the law of $Z$ and \eqref{equ:scalemart} for each of the two processes separately,
we obtain
\begin{equation}\label{equ:couplingbound}
  E|Z_{t}-X_{t}|=E\bigl[Z_{t}-X_{t}\bigr]
  \ \le\ M_{0}\bigl(E\,u(Z_{t})-E\,u(X_{t})\bigr)
  \ =\ M_{0}\bigl(u(x_{2})-u(x_{1})\bigr)
  \ \le\ \kappa\,(x_{2}-x_{1})
\end{equation}
for $0\le t\le T$.

Finally, we transfer this bound to the Picard iterates.  Let $w_{n}$ be the iterates \eqref{equ:EQ10} from the
proof of Theorem~\ref{thm:existence}.  By the induction principle, we show that
\begin{equation}\label{equ:lipinduction}
  |w_{n}(t,x_{2})-w_{n}(t,x_{1})|\ \le\ A_{0}e^{B_{0}t}\,|x_{2}-x_{1}|,
  \qquad 0\le t\le T,\ x_{1},x_{2}\in\mathbb{R}.
\end{equation}
For $n=0$ this holds because $w_{0}=g$ and $K_{g}\le A_{0}$, since $\kappa\ge1$ and
$e^{\Lambda T}>1$.  Suppose that \eqref{equ:lipinduction} holds for $n$, let $x_{1}<x_{2}$,
and let $(X,Z)$ be the coupling above.  Put
$a_{1}:=\int_{0}^{t}c\bigl(X_{s},0,w_{n}(t-s,X_{s})\bigr)ds$ and let $a_{2}$ be the same
expression with $Z$ in place of $X$.  Both are at most $\Lambda T$, so
\[
  \begin{aligned}
  |w_{n+1}(t,x_{2})-w_{n+1}(t,x_{1})|
  &\ \le\ E\bigl[|g(Z_{t})-g(X_{t})|\,e^{a_{2}}\bigr]
     +E\bigl[g(X_{t})\,|e^{a_{2}}-e^{a_{1}}|\bigr]\\
  &\ \le\ e^{\Lambda T}K_{g}\,E|Z_{t}-X_{t}|+\|g\|e^{\Lambda T}E|a_{2}-a_{1}| .
  \end{aligned}
\]
By the Lipschitz properties of $c$ in $x$ and in $w$ and by \eqref{equ:lipinduction},
\[
  |a_{2}-a_{1}|\ \le\ \int_{0}^{t}\bigl(K_{0}+K_{c}A_{0}e^{B_{0}(t-s)}\bigr)
     |Z_{s}-X_{s}|\,ds .
\]
Taking expectations and using \eqref{equ:couplingbound} at each time gives
\[
  \begin{aligned}
  |w_{n+1}(t,x_{2})-w_{n+1}(t,x_{1})|
  &\ \le\ \kappa\,e^{\Lambda T}\Bigl(K_{g}+\|g\|K_{0}t
     +\|g\|K_{c}\int_{0}^{t}A_{0}e^{B_{0}r}\,dr\Bigr)(x_{2}-x_{1})\\
  &\ \le\ \Bigl(A_{0}+B_{0}\int_{0}^{t}A_{0}e^{B_{0}r}\,dr\Bigr)(x_{2}-x_{1})
  \ =\ A_{0}e^{B_{0}t}(x_{2}-x_{1}),
  \end{aligned}
\]
which is \eqref{equ:lipinduction} for $n+1$.  Since $w_{n}\to w$ uniformly by the
proof of Theorem~\ref{thm:existence}, \eqref{equ:lipschyp} follows with
$K_{w}=A_{0}e^{B_{0}T}$.  For $w^{\delta}$ the argument is the same, with $X^{\delta}$
in place of $X$, since the bounds \eqref{equ:bilip} hold for $u^{\delta}$ and $v^{\delta}$
directly and the Picard iterates for \eqref{equ:EQ6} converge to $w^{\delta}$ by
the proof of Theorem~\ref{thm:existence}.
\end{proof}

\begin{corollary}\label{cor:holder}
Under the assumptions of Lemma~\ref{lem:lipschitz}, the generalized solution $w$ is
H\"older continuous in $t$ with exponent $1/2$, uniformly in $x$, that is, there is a
constant $C$ such that
\[
  |w(t+h,x)-w(t,x)|\ \le\ C\sqrt{h}
  \qquad\text{for all } x\in\mathbb{R},\ 0<t\le t+h\le T ,
\]
where $C=C(m_{0},M_{0},K_{g},K_{0},K_{c},\|g\|,\Lambda,\|c\|,T)$.  The same bound, with the same constant, holds for $w^{\delta}$ for every $\delta>0$.
\end{corollary}

\begin{proof}
By Lemma~\ref{lem:lipschitz}, \eqref{equ:lipschyp} holds with
$K_{w}=A_{0}e^{B_{0}T}$.  The hypothesis of Proposition~\ref{pro:holder} holds with
$m=m_{0}$ and $M=M_{0}$ on every compact set, so in its proof we may take
$b=\min\{M_{0}^{-1},2m_{0}\}$, independently of $K$.  The constant $C$ then does not
depend on $K$, and \eqref{equ:EQ21} holds uniformly in $x$.  For $w^{\delta}$, the proof of Proposition~\ref{pro:holder}
applies verbatim, since the bounds $du^{\delta}\ge M_{0}^{-1}dx$ and
$dv^{\delta}\ge2m_{0}\,dx$ hold for every $\delta>0$, and Lemma~\ref{lem:lipschitz}
gives the Lipschitz bound with the same constant.
\end{proof}

For Proposition~\ref{pro:convergence}, we need a slightly stronger form of the
weak convergences in Section~\ref{sec:setup}, in which the initial point may vary.
Assumption \textup{(W1)} below says that the narrow tube process converges to the
reduced one dimensional diffusion even when its initial point moves with
$\varepsilon$, as long as the longitudinal coordinate of the initial point converges.  Assumption
\textup{(W2)} says the same for the convergence of $X^{\delta}$ to $X$ as
$\delta\downarrow0$.

\medskip
\noindent\textbf{(W) (Weak convergence along convergent initial points.)}

\smallskip
\noindent\textup{(W1)} For every fixed $\delta>0$, every sequence
$\varepsilon_{k}\downarrow0$, and every sequence
$(x_{k},y_{k})\in G^{\varepsilon_{k},\delta}$ with $x_{k}\to x\in\mathbb{R}$, the law
of $(X^{\varepsilon_{k},\delta},Y^{\varepsilon_{k},\delta})$ under
$P^{\varepsilon_{k},\delta}_{x_{k},y_{k}}$ converges weakly in
$C([0,T];\mathbb{R}\times\mathbb{R}^{n})$ to the law of $(X^{\delta},0)$ under
$P^{\delta}_{x}$.

\smallskip
\noindent\textup{(W2)} For every sequence $\delta_{k}\downarrow0$ and every sequence
$x_{k}\to x\in\mathbb{R}$, the law of $X^{\delta_{k}}$ under $P^{\delta_{k}}_{x_{k}}$
converges weakly in $C([0,T];\mathbb{R})$ to the law of $X$ under $P_{x}$.

\medskip
Assumption \textup{(W)} concerns the convergence of the processes rather than
the nonlinear problem.  For a fixed initial point, \textup{(W1)} and
\textup{(W2)} are the weak convergences recalled in Section~\ref{sec:setup}, taken
from \cite{F96,FS} and \cite{FW2} (see also \cite{FW12} for bounded planar domains).
Their proofs also give convergence along convergent initial points.  For
\textup{(W1)}, the averaging estimates of \cite{FS} are coupling estimates whose
constants depend only on bounds for the geometry of the tube, and in the periodic
setting of Proposition~\ref{pro:convergence} these bounds are uniform over a period
cell, so the estimates hold uniformly in the initial point.  Together with the
continuous dependence of $X^{\delta}$ on its initial point, this gives \textup{(W1)}.  For \textup{(W2)}, a process $X^{\delta}$ started at a point close to $x$
reaches $x$ within a short time with high probability, by the formulas for exit
probabilities and expected exit times in terms of $u^{\delta}$ and $v^{\delta}$~\cite{EBD3,PM1},
since $u^{\delta}\to u$ locally uniformly with $u$ continuous.  After that time the
strong Markov property reduces the claim to the convergence from the fixed point
$x$.

\begin{proposition}\label{pro:convergence}
Under assumption \textup{(W)}, suppose that $G^{\delta}$ is
$1$-periodic in $x$ for every $\delta>0$ and that $g$ and $c$ are $1$-periodic in $x$.
Let $w^{\delta}$ be the solution of \eqref{equ:EQ2}, and let $w$ be the generalized
solution of \eqref{equ:EQ7}.  Then, for each $\delta>0$, the solutions
$w^{\varepsilon,\delta}$ of \eqref{equ:EQ1} satisfy
$w^{\varepsilon,\delta}(t,x,y)\to w^{\delta}(t,x)$ as $\varepsilon\downarrow0$, uniformly
in $t\in[0,T]$ and $(x,y)\in G^{\varepsilon,\delta}$.  Moreover, $w^{\delta}\to w$ as
$\delta\downarrow0$, uniformly on $[0,T]\times\mathbb{R}$.
\end{proposition}

\begin{proof}
We first take the limit $\varepsilon\downarrow0$ for fixed $\delta>0$.
Fix $\delta>0$ and let $w^{\delta}$ be the unique bounded solution of
\eqref{equ:EQ2}, which exists and is continuous and bounded on
$[0,T]\times\mathbb{R}$ by the remarks following \eqref{equ:EQ6}.  Set
\[
  \eta_{\varepsilon}(t):=\sup_{(x,y)\in G^{\varepsilon,\delta}}
    \bigl|w^{\varepsilon,\delta}(t,x,y)-w^{\delta}(t,x)\bigr|,
  \qquad 0\le t\le T,
\]
which is bounded, since both $w^{\varepsilon,\delta}$ and $w^{\delta}$ are bounded.
For a bounded function $\varphi$ on
$[0,T]\times G^{\varepsilon,\delta}$ define
\[
  \mathcal{F}^{\varepsilon}_{t}[\varphi](x,y)
  :=E^{\varepsilon,\delta}_{x,y}\left[g(X^{\varepsilon,\delta}_{t})
    \exp\left\{\int_{0}^{t}c\bigl(X^{\varepsilon,\delta}_{s},
      Y^{\varepsilon,\delta}_{s},
      \varphi(t-s,X^{\varepsilon,\delta}_{s},Y^{\varepsilon,\delta}_{s})\bigr)
      ds\right\}\right],
\]
so that \eqref{equ:EQ5} reads
$w^{\varepsilon,\delta}(t,\cdot)=\mathcal{F}^{\varepsilon}_{t}[w^{\varepsilon,\delta}]$.
Since $\mathcal{F}^{\varepsilon}_{t}$ acts on functions of $(t,x,y)$, write
$\bar w^{\delta}(t,x,y):=w^{\delta}(t,x)$ for the extension of $w^{\delta}$ to the
tube, constant in the transverse variable.  Then
\begin{equation}\label{equ:split}
  w^{\varepsilon,\delta}(t,x,y)-w^{\delta}(t,x)
  =\underbrace{\Bigl(\mathcal{F}^{\varepsilon}_{t}[w^{\varepsilon,\delta}](x,y)
      -\mathcal{F}^{\varepsilon}_{t}[\bar w^{\delta}](x,y)\Bigr)}_{=:\,\mathrm{I}}
  +\underbrace{\Bigl(\mathcal{F}^{\varepsilon}_{t}[\bar w^{\delta}](x,y)
      -w^{\delta}(t,x)\Bigr)}_{=:\,\mathrm{II}} .
\end{equation}
Term $\mathrm{I}$ involves a single process with two different arguments of the
nonlinearity, whereas $\mathrm{II}$ involves a single functional under two different
laws.

For the first term $\mathrm{I}$, the computation in the proof of
Lemma~\ref{lem:basic}, applied to $\mathcal{F}^{\varepsilon}_{t}$ in place of
$\mathcal{F}_{t}$, gives
\[
  |\mathrm{I}|
  \le\|g\|e^{\Lambda T}\,E^{\varepsilon,\delta}_{x,y}
     \int_{0}^{t}\bigl|c(\cdot,\cdot,w^{\varepsilon,\delta})
       -c(\cdot,\cdot,\bar w^{\delta})\bigr|\,ds
  \le\|g\|e^{\Lambda T}K_{c}\int_{0}^{t}\eta_{\varepsilon}(t-s)\,ds
  =K\int_{0}^{t}\eta_{\varepsilon}(r)\,dr,
\]
with $K$ as in Lemma~\ref{lem:basic}, finite and independent of $\varepsilon$ and
$t$.

For the second term $\mathrm{II}$, for a path
$(\xi,\zeta)\in C([0,T];\mathbb{R}\times\mathbb{R}^{n})$, define
\[
  \Theta_{t}(\xi,\zeta)
  :=g(\xi_{t})\exp\left\{\int_{0}^{t}
    c\bigl(\xi_{s},\zeta_{s},w^{\delta}(t-s,\xi_{s})\bigr)ds\right\}.
\]
Since $w^{\delta}$ solves \eqref{equ:EQ6}, we have
$w^{\delta}(t,x)=E^{\delta}_{x}\bigl[\Theta_{t}(X^{\delta},0)\bigr]$, and the same
functional evaluated along the process in the tube gives
$\mathcal{F}^{\varepsilon}_{t}[\bar w^{\delta}](x,y)
=E^{\varepsilon,\delta}_{x,y}\Theta_{t}
(X^{\varepsilon,\delta},Y^{\varepsilon,\delta})$.  Hence
\[
  \mathrm{II}=E^{\varepsilon,\delta}_{x,y}\Theta_{t}
  (X^{\varepsilon,\delta},Y^{\varepsilon,\delta})
  -E^{\delta}_{x}\Theta_{t}(X^{\delta},0).
\]
The functional $\Theta_{t}$ is bounded by $\|g\|e^{\Lambda T}$.  We show that
\[
  R_{\varepsilon}:=\sup_{0\le t\le T}\ \sup_{(x,y)\in G^{\varepsilon,\delta}}
    |\mathrm{II}(t,x,y)|\ \longrightarrow\ 0
  \qquad\text{as }\varepsilon\downarrow0 .
\]
First, we reduce the supremum to $x\in[0,1]$.  Since $g$, $c$, and $V^{\delta}$ are
$1$-periodic, $w^{\delta}$ is $1$-periodic in $x$ by the uniqueness of the solution of
\eqref{equ:EQ2}, and hence $\Theta_{t}(\xi+1,\zeta)=\Theta_{t}(\xi,\zeta)$.  Since the
tube is $1$-periodic, the law of $(X^{\varepsilon,\delta},Y^{\varepsilon,\delta})$
under $P^{\varepsilon,\delta}_{x+1,y}$ is the law of
$(X^{\varepsilon,\delta}+1,Y^{\varepsilon,\delta})$ under
$P^{\varepsilon,\delta}_{x,y}$.  Similarly, since $V^{\delta}$ is $1$-periodic, the law
of $X^{\delta}$ under $P^{\delta}_{x+1}$ is the law of $X^{\delta}+1$ under
$P^{\delta}_{x}$.  Therefore
$\mathrm{II}(t,x+1,y)=\mathrm{II}(t,x,y)$, and it suffices to take the supremum over
$x\in[0,1]$.

Next, $\Theta_{t}(\xi,\zeta)$ is jointly continuous in $t\in[0,T]$ and
$(\xi,\zeta)\in C([0,T];\mathbb{R}\times\mathbb{R}^{n})$, where the path space carries
the topology of uniform convergence.  If we write the integral in $\Theta_{t}$ as
\[
  \int_{0}^{t}c\bigl(\xi_{t-r},\zeta_{t-r},w^{\delta}(r,\xi_{t-r})\bigr)\,dr,
\]
then
dominated convergence gives $\Theta_{t_{m}}(\xi^{(m)},\zeta^{(m)})\to\Theta_{t}(\xi,\zeta)$
whenever $t_{m}\to t$ and $(\xi^{(m)},\zeta^{(m)})\to(\xi,\zeta)$ uniformly on
$[0,T]$.  This uses only the continuity of $w^{\delta}(r,\cdot)$ for each $r$.

Finally, suppose that $R_{\varepsilon}$ does not tend to $0$.  Then there are $c_{*}>0$,
$\varepsilon_{k}\downarrow0$, and $(t_{k},x_{k},y_{k})$ with $t_{k}\in[0,T]$,
$x_{k}\in[0,1]$, and $(x_{k},y_{k})\in G^{\varepsilon_{k},\delta}$ such that
$|\mathrm{II}(t_{k},x_{k},y_{k})|\ge c_{*}$, where $\mathrm{II}$ is computed with
$\varepsilon=\varepsilon_{k}$.  After passing to a subsequence, $t_{k}\to t$ and
$x_{k}\to x$.  By \textup{(W1)} and the extended continuous mapping theorem
\cite[Theorem~5.5]{PB},
\[
  E^{\varepsilon_{k},\delta}_{x_{k},y_{k}}\Theta_{t_{k}}
    \bigl(X^{\varepsilon_{k},\delta},Y^{\varepsilon_{k},\delta}\bigr)
  \ \longrightarrow\ E^{\delta}_{x}\Theta_{t}\bigl(X^{\delta},0\bigr),
\]
and by the Feller property of $X^{\delta}$ and the same theorem,
\[
  E^{\delta}_{x_{k}}\Theta_{t_{k}}\bigl(X^{\delta},0\bigr)
  \ \longrightarrow\ E^{\delta}_{x}\Theta_{t}\bigl(X^{\delta},0\bigr).
\]
Hence $\mathrm{II}(t_{k},x_{k},y_{k})\to0$, which contradicts
$|\mathrm{II}(t_{k},x_{k},y_{k})|\ge c_{*}$.

Taking the supremum over $(x,y)$ in \eqref{equ:split} and combining the two bounds,
\[
  \eta_{\varepsilon}(t)\ \le\ K\int_{0}^{t}\eta_{\varepsilon}(r)\,dr
    +R_{\varepsilon},\qquad 0\le t\le T .
\]
Since $\eta_{\varepsilon}$ is bounded and measurable, Gr\"onwall's inequality yields
\[
  \sup_{0\le t\le T}\eta_{\varepsilon}(t)\ \le\ R_{\varepsilon}\,e^{KT}
  \ \longrightarrow\ 0
  \qquad\text{as }\varepsilon\downarrow0 ,
\]
so $w^{\varepsilon,\delta}\to w^{\delta}$ uniformly, for each fixed $\delta>0$.

\medskip
It remains to take the limit $\delta\downarrow0$.
The generator of \eqref{equ:EQ2} is
$\mathcal{L}^{\delta}f=\tfrac12 f_{xx}+\tfrac12\bigl(\ln V^{\delta}\bigr)'f_{x}
=D_{v^{\delta}}D_{u^{\delta}}f$, with $u^{\delta}$ and $v^{\delta}$ as in
Section~\ref{sec:setup}.  Since $(u^{\delta},v^{\delta})$ converges to $(u,v)$ in the
sense stated there, the theorem of Freidlin and Wentzell \cite{FW2} on weak
convergence of one dimensional Markov processes shows that $X^{\delta}_{t}$ converges
weakly in $C([0,T];\mathbb{R})$ to the process $X_{t}$ generated by $D_{v}D_{u}$.

Let $w$ be the generalized solution of \eqref{equ:EQ7} given by
Theorem~\ref{thm:existence} and set
\[
  \eta_{\delta}(t):=\sup_{x\in\mathbb{R}}\bigl|w^{\delta}(t,x)-w(t,x)\bigr|,
  \qquad 0\le t\le T,
\]
which is bounded.  Let $\mathcal{F}^{\delta}_{t}$ be defined as in \eqref{equ:Fop}
with $X^{\delta}$ in place of $X$, so that \eqref{equ:EQ6} reads
$w^{\delta}(t,\cdot)=\mathcal{F}^{\delta}_{t}[w^{\delta}]$, and write
\[
  w^{\delta}(t,x)-w(t,x)
  =\underbrace{\Bigl(\mathcal{F}^{\delta}_{t}[w^{\delta}](x)
      -\mathcal{F}^{\delta}_{t}[w](x)\Bigr)}_{=:\,\mathrm{I}}
  +\underbrace{\Bigl(\mathcal{F}^{\delta}_{t}[w](x)-w(t,x)\Bigr)}_{=:\,\mathrm{II}} ,
\]
which is \eqref{equ:split} with $X^{\delta}$ in place of
$(X^{\varepsilon,\delta},Y^{\varepsilon,\delta})$ and $w$ in place of $w^{\delta}$.
No extension is needed here, since both functions are defined on
$[0,T]\times\mathbb{R}$.

Term $\mathrm{I}$ is estimated as before and contributes
$K\int_{0}^{t}\eta_{\delta}(r)\,dr$ with the same constant $K$.  For term
$\mathrm{II}$ the relevant functional is
\[
  \Theta_{t}(\xi):=g(\xi_{t})\exp\left\{\int_{0}^{t}
    c\bigl(\xi_{s},0,w(t-s,\xi_{s})\bigr)ds\right\},
  \qquad \xi\in C([0,T];\mathbb{R}),
\]
which is bounded and, since $w(r,\cdot)$ is continuous for each $r$ by
Definition~\ref{def:gensol}, continuous.  Here $u$ and $v$ have $1$-periodic increments, as one sees by letting
$\delta\downarrow0$ in the same property of $u^{\delta}$ and $v^{\delta}$ and using the
right continuity of $v$.  Hence the law of $X$ under $P_{x+1}$ is the law of $X+1$ under
$P_{x}$, and $w$ is $1$-periodic in $x$ by the uniqueness part of
Theorem~\ref{thm:existence}.  The argument used for
$R_{\varepsilon}$, with \textup{(W2)} in place of \textup{(W1)} and the Feller
property of $X$ in place of that of $X^{\delta}$, therefore gives
\[
  R_{\delta}:=\sup_{0\le t\le T}\ \sup_{x\in\mathbb{R}}|\mathrm{II}(t,x)|
  \ \longrightarrow\ 0
  \qquad\text{as }\delta\downarrow0 ,
\]
and Gr\"onwall gives
$\sup_{0\le t\le T}\eta_{\delta}(t)\le R_{\delta}e^{KT}\to0$, so
$w^{\delta}\to w$ uniformly on $[0,T]\times\mathbb{R}$.
\end{proof}

\section{Discussion}\label{sec:remarks}

The results above concern the Cauchy problem on a finite time interval $[0,T]$.  The
question that originally motivated them, the asymptotic speed of wave front
propagation, is a statement about $t\to\infty$.

Suppose that the medium is periodic, that is, that $\partial G^{\delta}$ has period
one in $x$.  For KPP nonlinearities the front position is governed by the linearized
problem, with potential $c_{0}$ from Section~\ref{sec:setup}, through a family of
twisted operators
\[
  L^{z,\varepsilon,\delta}
  =\tfrac12\Delta-\langle z,\nabla\rangle+c_{0}(x,y)+\tfrac12\langle z,z\rangle ,
  \qquad z=(z_{1},\mathbf{0})\in\mathbb{R}\times\mathbb{R}^{n},
\]
acting on functions periodic in $x$ with period one and satisfying the Neumann
condition of \eqref{equ:EQ1}.  Writing $\lambda^{\varepsilon,\delta}_{1}(z)$ for the
principal eigenvalue of $L^{z,\varepsilon,\delta}$, and
$\lambda^{\delta}_{1}(z_{1})$ for the principal eigenvalue of the corresponding
one dimensional operator associated with \eqref{equ:EQ2}, the asymptotic front speed
in each case is obtained from the eigenvalue by a Legendre-type variational formula~\cite{BHN05,FG79}.
Convergence of the speeds would follow from
\begin{equation}\label{equ:openproblem}
  \lambda^{\varepsilon,\delta}_{1}(z)\ \longrightarrow\ \lambda^{\delta}_{1}(z_{1})
  \quad(\varepsilon\downarrow0),
  \qquad
  \lambda^{\delta}_{1}(z_{1})\ \longrightarrow\ \lambda_{1}(z_{1})
  \quad(\delta\downarrow0),
\end{equation}
where $\lambda_{1}$ is the principal eigenvalue of the twisted operator built from
$D_{v}D_{u}$.

We have not yet been able to establish \eqref{equ:openproblem}, which we leave for future work. 
The difficulty is that, by the variational formula of Donsker and
Varadhan \cite{DV}, the principal eigenvalue is a $t\to\infty$ logarithmic growth
rate of the Feynman--Kac semigroup, whereas Proposition~\ref{pro:convergence} and the
weak convergence it rests on control the semigroup only for each fixed $t$, with no
uniformity as $t\to\infty$.  An analytic
approach would use that the period cell is bounded, so the twisted operator has
compact resolvent and discrete spectrum, the principal eigenvalue is simple and
isolated by the Krein--Rutman theorem, and convergence of isolated simple eigenvalues
would follow from resolvent convergence by the perturbation theory of Kato
\cite{TK}.  A variational approach would establish upper and lower semicontinuity in
$\varepsilon$ of the Donsker--Varadhan formula separately.  Of the two limits in \eqref{equ:openproblem}, the one as
$\delta\downarrow0$ appears the harder, since the limiting operator $D_{v}D_{u}$ is a
Feller generalized operator whose domain is not a classical function space.

The conditions under which the front jumps, described by Freidlin and Spiliopoulos
\cite{FS}, suggest that \eqref{equ:openproblem} may fail in general, and that
identifying what separates convergence from discontinuity of the speed is the
substantive question.

\end{document}